\documentclass[preprint,12pt]{elsarticle1}
\biboptions{numbers,sort&compress}

\usepackage{graphicx}
\usepackage{epstopdf}
\usepackage{amssymb}
\usepackage{amsthm}
\usepackage{amsmath}
\usepackage{epic}
\usepackage{setspace}
\usepackage{bm}
\usepackage{xcolor}
\usepackage[hidelinks]{hyperref}
\usepackage{enumitem}
\newtheorem{theorem}{Theorem}[section]

\newtheorem{lemma}[theorem]{Lemma}

\newtheorem{definition}[theorem]{Definition}

\journal{~}
\begin{document}

\begin{spacing}{1.15}

\begin{frontmatter}
\title{\textbf{\\
Clique spectral extremal problem on disjoint color-critical graphs}}

\author[label1]{Changjiang Bu\corref{cor}}\ead{buchangjiang@hrbeu.edu.cn}
\author[label1]{Peiyan Wei}\ead{weipeiyan0705@126.com}
\author[label1]{Haotian Zeng}\ead{zenghaotian810@163.com}

\cortext[cor]{Corresponding author}

\address{
\address[label1]{School of  Mathematical Sciences, Harbin Engineering University, Harbin 150001, PR China}
}

\begin{abstract}
For a given graph $F$,
a graph $G$ is called $F$-free if it does not contain $F$ as a subgraph.
A graph is color-critical if  deleting one of its edges decreases its chromatic number.
Let $F_1, F_2, \cdots, F_t$ be $t$ disjoint color-critical graphs with  chromatic number $r+1$. For $2 \leq s \leq r$ and sufficiently large $n$, we determine the unique extremal graph with the maximum $s$-clique spectral radius among all $n$-vertex $\bigcup_{i=1}^t F_i$-free graphs.

\end{abstract}

\end{frontmatter}

\section{Introduction}
An \textit{$s$-clique} of a graph $G$
is a subset of $s$ vertices that induces a complete subgraph of $G$.
The set of all $s$-cliques in $G$ is denoted by $C_s(G)$. 
A graph $G$ is \textit{color-critical} if there exists an edge $e$ such that $\chi(G-e)<\chi (G)$,
where $\chi(G)$ is the chromatic number of $G$.
Let $K(n_1,n_2,\dots,n_r)$ be the \textit{complete $r$-partite graph} with parts of sizes $n_1,n_2,\dots,n_r$.
An \textit{$r$-partite Tur\'{a}n graph} $T_r(n)$ is a complete $r$-partite graph $K(n_1,n_2,\dots,n_r)$ on $n$ vertices with $|n_i - n_j| \leq 1$ for all $i,j \in \{1,2,\dots,r\}$.
Let $K^+(n_1,n_2,\dots,n_r)\ (n_1\geq 2)$ be the graph obtained from $K(n_1,n_2,\dots,n_r)$ by adding an edge to the first part.
Let $kG$ denote the \textit{disjoint union} of $k$ copies of $G$.
For two disjoint graphs $G_1$ and $G_2$,
the \textit{join} $G_1+G_2$ is the graph obtained from $G_1\cup G_2$ by adding all edges between the vertices of $G_1$ and $G_2$.

For a given graph $F$,
we say that $G$ is \textit{$F$-free} if it does not contain $F$ as a subgraph.
The maximum number of edges among all $F$-free graphs on $n$ vertices is called the \textit{Tur\'{a}n number of $F$},
denoted by ${\rm ex}(n, F)$,
and the corresponding graph is called an \textit{extremal graph for $F$}.
There are lots of researches on Tur\'{a}n-type problem,
see \cite{Turan} for complete graph,
\cite{Moon} for disjoint union of $t$ copies of complete graph and \cite{Fang} for disjoint union of $t$ copies of odd cycle.
Clearly,
complete graph and odd cycle are color-critical graphs.
Simonovits \cite{Simonovits} determined the extremal graph for disjoint union of color-critical graphs.
\begin{theorem}{\rm (\cite{Simonovits})}\label{Simonovits}
Let $F_1,F_2,\dots,F_t$ be $t$ disjoint color-critical graphs with chromatic number $r+1$.
Then $K_{t-1}+T_r(n-t+1)$ is the unique graph with respect to ${\rm ex}(n,\bigcup_{i=1}^tF_i)$ for sufficiently large $n$.
\end{theorem}

Nikiforov \cite{Nikiforov2010} studied the \textit{spectral Tur\'{a}n-type problem},
i.e.,
for a given graph $F$,
determining the maximum spectral radius among all $n$-vertex graphs containing no $F$ as a subgraph.
Let ${\rm spex}(n,F)$ denote the maximum spectral radius among all $F$-free graphs on $n$ vertices.
Researches on spectral Tur\'{a}n-type problem has attracted increasing interest,
see \cite{Nikiforov2007} for complete graph,
\cite{Nikiforov2008} for odd cycle,
\cite{Zhai2023} for book graph,
\cite{Ni} for disjoint union of $t$ copies of complete graph and \cite{Fang} for disjoint union of t copies of odd cycle. 
In 2024,
Lei and Li \cite{Lei} gave the spectral version of Theorem \ref{Simonovits}.
\begin{theorem}{\rm (\cite{Lei})}\label{Leixingyu}
Let $F_1,F_2,\dots,F_t$ be $t$ disjoint color-critical graphs with chromatic number $r+1$.
Then $K_{t-1}+T_r(n-t+1)$ is the unique graph with respect to ${\rm spex}(n,\bigcup_{i=1}^tF_i)$ for sufficiently large $n$.
\end{theorem}
The clique spectral extremal problem is a spectral version of the generalized Tur\'{a}n number ${\rm ex}(n,K_s,F)$,
where ${\rm ex}(n,K_s,F)$ is the maximum number of $K_s$ among all $F$-free graphs on $n$ vertices for given graphs $K_s$ and $F$.
In particular, 
when $s=2$,
clique spectral extremal problem reduces to spectral Tur\'{a}n-type problem.
The clique tensor of a graph and its spectral radius  were introduced by Liu and Bu \cite{Liu2}. 
Let ${\rm spex}_s(n,F)$ denote the maximum $s$-clique spectral radius among all $F$-free graphs on $n$ vertices.
In 2023,
Liu and Bu \cite{Liu2} determined that $T_r(n)$ is the unique graph with respect to ${\rm spex}_r(n,K_{r+1})$.
In 2025,
Yu and Peng \cite{Yu2025} further determined that ${\rm spex}_s(n,K_{r+1})=\rho_s(T_r(n))$ for $2\leq s\leq r$,
and $T_r(n)$ is the unique extremal graph.
In 2026,
Xu and Fan \cite{Xu} determined that $K_{t-1}+T_r(n-t+1)$ is the unique graph with respect to ${\rm spex}_s(n,tK_{r+1})$ for $2\leq s\leq r$.
For more results on clique spectral extremal problem,
see \cite{Yu,Yan,Liu1,Liu3}.

In this paper,
we focus on the maximum $s$-clique spectral radius among all $\bigcup_{i=1}^tF_i$-free graphs,
where $F_1, F_2, \dots,F_t$ are $t$ disjoint color-critical graphs with chromatic number $r+1$.
Firstly,
we determine the graph with respect to ${\rm spex}_s(n,tK^+(n_1,n_2,\dots,n_r))$.

\begin{theorem}\label{theorem1}
Given positive integers $s,t,r,n_1,n_2,\dots,n_r$ with $2\leq s\leq r,t\geq 2$ and $n_1\geq 2$.
Then $K_{t-1}+T_r(n-t+1)$ is the unique graph with respect to ${\rm spex}_s(n,tK^+(n_1,n_2,\dots,n_r))$ for sufficiently large $n$.
\end{theorem}

Clearly,
for an arbitrary color-critical graph $F$ with chromatic number $r+1$,
there exist positive integers $\hat{n}_1,\hat{n}_2,\dots,\hat{n}_r$ such that $F$ is a subgraph of $K^+(\hat{n}_1,\hat{n}_2,\dots,\hat{n}_r)$.
We obtain the following result. 
\begin{theorem}\label{theorem2} 
Given positive integers $s,r,t$ with $2\leq s\leq r$,
$t\geq2$.
Let $F_1,F_2,\dots,F_t$ be $t$ disjoint color-critical graphs with chromatic number $r+1$.
Then $K_{t-1}+T_r(n-t+1)$ is the unique graph with respect to ${\rm spex}_s(n,\bigcup_{i=1}^tF_i)$ for sufficiently large $n$.
\end{theorem}

\section{Preliminaries}
In this section,
we will introduce some necessary lemmas for the proofs of the main theorems.

Let $\mathbb{C}^{[k,n]}$ be the set of $k$-order $n$-dimensional complex tensors.
For $\mathcal{A}=(a_{i_1i_2\dots i_k})\in \mathbb{C}^{[k,n]}$ and an $n$-dimensional complex vector $x=(x_1,x_2,\dots,x_n)^{\top}$,
the $\mathcal{A}x^{k-1}$ is an $n$-dimensional complex vector whose $i$-th component is $$(\mathcal{A}x^{k-1})_i=\sum_{i_2,\dots,i_k=1}^na_{ii_2\dots i_k}x_{i_2}\dots x_{i_k},$$
where $x^{[k-1]}=(x_1^{k-1},x_2^{k-1},\dots,x_n^{k-1})^{\top}$.
If there exist a complex number $\lambda$ and a nonzero vector $x$ such that $$\mathcal{A}x^{k-1}=\lambda x^{[k-1]},$$
then $\lambda$ is an \textit{eigenvalue} of $\mathcal{A}$ and $x$ is an \textit{eigenvector} of $\mathcal{A}$ corresponding to $\lambda$  \cite{Qi2005,Lim}.
The maximal modulus of all eigenvalues of $\mathcal{A}$ is called the spectral radius of $\mathcal{A}$,
denoted by $\rho(\mathcal{A})$.
The tensor $\mathcal{A}$ is called symmetric if its entries are invariant under any permutation of their indices.
If all elements of $\mathcal{A}$ are nonnegative,
then $\mathcal{A}$ is called the nonnegative tensor.
Qi \cite{Qi1} gave the following result for nonnegative symmetric tensors.
\begin{lemma}\label{Lemma1}{\rm (\cite{Qi1})}
Suppose that $\mathcal{A}$ is a $k$-order   $n$-dimensional symmetric nonnegative tensor.
Then
$$\rho(\mathcal{A})=\max\left\{x^{\top}\mathcal{A}x^{k-1}\ |\  x=(x_1,\dots,x_n)^{\top}\in \mathbb{R}^n_+,\sum_{i=1}^nx_i^k=1\right\},$$
where $\mathbb{R}_ + ^n$ is the set of $n$-dimensional nonnegative real vectors.
\end{lemma}

Liu and Bu \cite{Liu2} introduced the definition of the $s$-clique tensor.
\begin{definition}\label{s-clique}{\rm (\cite{Liu2})}
Let $G$ be a graph on $n$ vertices.
The tensor $\mathcal{A}_s(G)=(a_{i_1i_2\dots i_s})$ is called the $s$-clique tensor of $G$,
where
$$a_{i_1i_2\dots i_s} = \begin{cases}
\frac{1}{(s-1)!}, & \{i_1,i_2, \dots, i_s\} \in C_s(G), \\
0, & otherwise.
\end{cases}$$
\end{definition}
The eigenvalue of the $s$-clique tensor is called the $s$-clique eigenvalue of $G$.
The spectral radius of the $s$-clique tensor is called the \textit{$s$-clique spectral radius} of $G$,
denoted by $\rho_s(G)$.
When $s=2$,
$\mathcal{A}_2(G)$ is the adjacency matrix and $\rho_2(G)$ is the spectral radius of $G$.

For two $k$-order $n$-dimensional nonnegative tensors $\mathcal{A}$ and $\mathcal{B}$,
if $\mathcal{A}-\mathcal{B}$ is nonnegative,
we write $\mathcal{A}\geq \mathcal{B}$.
\begin{lemma}\label{Lemma3}{\rm (\cite{Khan1})}
Suppose $\mathcal{A}\geq \mathcal{B} \geq0$, then $\rho({\mathcal A}) \geq \rho ({\mathcal B})$.
Furthermore, if $\mathcal{A}$ is weakly irreducible and $\mathcal{A} \neq \mathcal{B}$,
then $\rho ({\mathcal A}) > \rho ({\mathcal B})$.
\end{lemma}

For $\mathcal{A}=(a_{i_1i_2\dots i_k}) \in \mathbb{C}^{[k,n]}$,
let $\Gamma_{\mathcal{A}}=(V_{\mathcal{A}},E_{\mathcal{A}})$ be the associated digraph of $\mathcal{A}$ with the vertex set $V_{\mathcal{A}}=\{1,2,\dots,n\}$ and the arc set $E_{\mathcal{A}}=\{(i,j):a_{ii_2\dots i_k}\neq0,j\in\{i_2,\dots,i_k\}\}$.
$\Gamma_{\mathcal{A}}$ is called strongly connected if there exists a directed path from $i$ to $j$ in $\Gamma_{\mathcal{A}}$ for any two distinct $i, j \in V_{\mathcal{A}}$.
A nonnegative tensor $\mathcal{A}$ is weakly irreducible if and only if $\Gamma_{\mathcal{A}}$ is strongly connected \cite{Friedland1}.

\begin{lemma}\label{Lemma2}
{\rm (1). (\cite{Chang})} If $\mathcal{A}$ is a $k$-order  $n$-dimensional  nonnegative tensor,
then $\rho(\mathcal{A})$ is an eigenvalue of $\mathcal{A}$,
with a nonnegative eigenvector corresponding to it;\\
{\rm (2). (\cite{Friedland1})} If furthermore $\mathcal{A}$ is weakly irreducible, then $\rho ({\mathcal A})$ is the unique positive eigenvalue of $\mathcal{A}$ corresponding to a unique positive eigenvector up to scaling.
\end{lemma}

A tensor can be permuted into a block lower triangular tensor,
and its eigenvalues remain invariant under such a permutation.

\begin{lemma}\label{Lemma4}{\rm (\cite{Shao1})}
Let $\mathcal{A}$ be a $k$-order $n$-dimensional tensor.
Then there exist positive integers $\ell$ and $n_1,\dots,n_{\ell}$ with $\sum_{i=1}^{\ell}n_i=n$ such that $\mathcal{A}$ can be permuted into an $(n_1,\dots,n_{\ell})$-lower triangular block tensor,
whose diagonal blocks ${\mathcal A_1},\dots,{\mathcal A_{\ell}}$ are weakly irreducible.
Furthermore, we have $\rho ({\mathcal A})=max\{\rho ({\mathcal A_1}),\dots,\rho({\mathcal A_{\ell}})\}$.
\end{lemma}

An \textit{$s$-clique walk} of a graph $G$ is a sequence of $s$-cliques $c_s^{(1)},c_s^{(2)},\dots, c_s^{(m)}$ in which $s$-cliques $c_s^{(i)}$ and $c_s^{(i+1)}$ have at least one vertex in common for $i = 1, \dots , m -1$.
If any two vertices of a graph are connected to each other by $s$-clique walks,
then the graph is called \textit{$s$-clique connected} \cite{Liu2}.
Let $\tilde{G}$ be the subgraph of a graph $G$ obtained by deleting all edges that are not contained in any $s$-clique.
The connected components in $\tilde{G}$ are called \textit{$s$-clique connected components} of $G$.
If $G$ is not $s$-clique connected,
by Lemma \ref{Lemma4},
there exists an $s$-clique connected component $H$ of $G$ such that $\rho_s(G)=\rho_s(H)$.

\begin{lemma}\label{dangqiejindang}{\rm (\cite{Liu2})}
The $s$-clique tensor of a graph $G$ is weakly irreducible if and only if $G$ is $s$-clique connected.
\end{lemma}

The following result gives a sharp upper bound on the number of $s$-cliques in $G$. 
\begin{lemma}\label{lemma6}{\rm (\cite{Liu2})}
Let $G$ be a graph with $n$ vertices.
Then $$|C_s(G)| \leq \frac{n}{s}\rho_s(G).$$
Furthermore, the equality holds if the number of $s$-cliques containing $i$ is equal for all $i=1,2,\dots,n$.
\end{lemma}

Liu and Bu \cite{Liu3} established the following spectral analogue of the Erd\H{o}s-Simonovits stability theorem for $s$-clique tensors.
\begin{lemma}\label{liuESS}{\rm (\cite{Liu3})}
Let $H$ be a graph with $\chi(H) = r+1>s\geq2$.
For every $\epsilon > 0$,
there exist $\delta > 0$ and $n_0$ such that if $G$ is an $H$-free graph on $n\geq n_0$ vertices and $$\rho_s(G)\geq \left({r-1\choose s-1}\left(\frac{1}{r}\right)^{s-1}-\delta\right)n^{s-1},$$
then $G$ can be obtained from $T_r(n)$ by adding and deleting at most $\epsilon n^2$ edges.
\end{lemma}

Let $\mu(G)$ and $\Delta(G)$ be the matching number and maximum degree of $G$,
respectively.
For two integers $\mu$ and $\Delta$,
let $f(\mu,\Delta)=\max\{e(G) \ |\ \mu(G)\leq \mu,\Delta(G)\leq\Delta\}$.
In 1976,
Chv\'{a}tal-Hanson obtained the following result.
\begin{lemma}{\rm (\cite{Hanson})}\label{Hanson}
For any two integers $\mu\geq 1$ and $\Delta\geq 1$,
we have $$f(\mu,\Delta)=\Delta\mu+\left\lfloor\frac{\Delta}{2}\right\rfloor\left\lfloor\frac{\mu}{\lceil\Delta/2\rceil}\right\rfloor\leq \mu(\Delta+1).$$
\end{lemma}

\begin{lemma}{\rm (\cite{Cioaba})}\label{Cioaba}
For $n$ finite sets $V_1,\dots,V_n$,
we have $|V_1\cap\dots\cap V_n|\geq \sum_{i=1}^n|V_i|-(n-1)|\bigcup_{i=1}^nV_i|$.
\end{lemma}

The following lemma states that, for $K_{t-1}+K(n_1,n_2,\ldots,n_r)$, the
$s$-clique spectral radius is increased by balancing the part sizes of its
complete $r$-partite part.

\begin{lemma}{\rm (\cite{Xu})}\label{Xu}
Let $2\leq s\leq r$.
For a complete $r$-partite graph $K(n_1,n_2,\dots,n_r)$, 
if there exist $i$ and $j$ with $n_i-n_j\geq 2$,
then $$\rho_s\big(K_{t-1}+K(n_1,\dots,n_i-1,\dots,n_j+1,\dots,n_r)\big)>\rho_s\big(K_{t-1}+K(n_1,\dots,n_i,\dots,n_j,\dots,n_r)\big).$$
\end{lemma}

\section{The proofs of Theorems \ref{theorem1} and \ref{theorem2} }
In this section,
we give the proofs of Theorems \ref{theorem1} and \ref{theorem2}.
First of all,
for sufficiently large $n$,
we assume that $G^*$ is a graph  attaining the maximum $s$-clique spectral radius among all $tK^+(n_1,n_2,\dots,n_r)$-free graphs on $n$ vertices where $t\geq2$, 
$2\leq s\leq r$ and let $\sum_{i=1}^rn_i=h$.
Our aim is to characterize the structure of $G^*$ and show that $G^*$ is $K_{t-1}+T_r(n-t+1)$. 

The proof of Theorem \ref{theorem1} is outlined as follows.
\begin{itemize}
\item We give a lower bound on $\rho_s(G^*)$,
see Lemma \ref{311}.
By Lemmas \ref{liuESS}, \ref{311} and \ref{lei},
we obtain an $r$-partition $V(G^*)=V_1\cup \dots\cup V_r$ such that $\frac{n}{r}-2\xi^2n<|V_i|<\frac{n}{r}+2\xi^2n$ for each $i\in \{1,2,
\dots,r\}$.

\item We show that any edge in $G^*[V_i\setminus (L\cup W\cup S)]$ and any vertices $w\in W$ can be extended to a copy of $K^+(n_1,n_2,\dots,n_r)$,
see Lemmas \ref{344} and \ref{3112}.

\item We show some local structural properties of $G^*$,
see Lemmas \ref{377}, \ref{388}, \ref{3991}, \ref{312} and \ref{355}. 
And show that $G^*$ is $tK_{r+1}$-free,
see Lemma \ref{3221}.

\item After we prove that $G^*$ is $tK_{r+1}$-free,
we show that $G^*$ is $s$-clique connected,
see Lemma \ref{322} and give the exact characterization of $W$,
see Lemmas \ref{313} and \ref{314}.
Finally we show that $G^*$ is $K_{t-1}+T_r(n-t+1)$.
\end{itemize}

\begin{lemma}\label{311}
$\rho_s(G^*)\geq {r-1 \choose s-1}\left(\frac{n}{r}\right)^{s-1}+cn^{s-2}$,
where $c$ is a positive constant.
\end{lemma}

\begin{proof}
Since $G^*$ has the maximum $s$-clique spectral radius and, 
clearly, 
$K_{t-1}+T_r(n-t+1)$ is $tK^+(n_1,n_2,\dots,n_r)$-free,
we have 
\begin{equation}\label{0}
\rho_s(G^*)\geq \rho_s(K_{t-1}+T_r(n-t+1)).
\end{equation}
By Lemma \ref{lemma6}, 
we have $\rho_s(K_{t-1}+T_r(n-t+1))\geq \frac{s}{n}|C_s(K_{t-1}+T_r(n-t+1))|$.
It remains to estimate the number of $s$-cliques in $K_{t-1}+T_r(n-t+1)$.
Since an $s$-clique contains either no vertex of $K_{t-1}$,
one vertex of $K_{t-1}$, or at least two vertices of $K_{t-1}$,
we have
\begin{align*}
|C_s\big(K_{t-1}+T_r(n-t+1)\big)|
&= |C_{s}(T_r(n-t+1))|+(t-1)|C_{s-1}(T_r(n-t+1))|+O(n^{s-2}).
\end{align*}
Since $|C_{s}(T_r(n-t+1))|=\binom{r}{s} \left( \frac{n - t + 1}{r} \right)^s + O(n^{s-2})=\binom{r}{s} \left( \frac{n}{r} \right)^s - (t-1) \binom{r-1}{s-1} \left( \frac{n}{r} \right)^{s-1} + O(n^{s-2})$ and $|C_{s-1}(T_r(n-t+1))|=\binom{r}{s-1} \left( \frac{n}{r} \right)^{s-1} + O(n^{s-2})$,
we have 
\begin{align*}
|C_s\big(K_{t-1}+T_r(n-t+1)\big)|&=\binom{r}{s} \left( \frac{n}{r} \right)^s + (t-1) \binom{r-1}{s-2} \left( \frac{n}{r} \right)^{s-1} + O(n^{s-2}).
\end{align*}
Hence,
for sufficiently large $n$,
there exists a  positive constant $c$ such that 
$$\frac{s}{n}|C_s(K_{t-1}+ T_r(n-t+1))|\ge\binom{r-1}{s-1}\left(\frac nr\right)^{s-1}+c n^{s-2}.$$
Therefore,
by Eq. (\ref{0}), we have
$\rho_s(G^*)
\geq{r-1\choose s-1}\left(\frac{n}{r}\right)^{s-1}+cn^{s-2}$.
\end{proof}

Since $\chi(tK^+(n_1,n_2,\dots,n_r))=r+1$,
by Lemmas \ref{liuESS} and \ref{311},
for every $\epsilon>0$,
$G^*$ can be obtained from $T_r(n)$ by adding and deleting at most $\epsilon n^2$ edges.
For disjoint $U,V\subseteq V(G)$, 
let $E(U,V)$ denote the set of edges of $G$ between $U$ and $V$,
and $e(U,V)=|E(U,V)|$,
$\overline e(U,V)=|U||V|-e(U,V).$
For $U\subseteq V(G)$ and $v\in V(G)$, 
let $G[U]$ denote the subgraph of $G$ induced by $U$
and $d_U(v)=|N(v)\cap U|$.
For a $tK^+(n_1,n_2,\dots,n_r)$-free graph $G$ on $n$ vertices,
Lei and Li \cite{Lei} proved the following results.
\begin{lemma}{\rm  (\cite{Lei})}\label{lei}
If $G$ is a $tK^+(n_1,n_2,\dots,n_r)$-free graph obtained from $T_r(n)$ by adding and deleting at most $\epsilon n^2$ edges,
where $\epsilon<\frac{1}{64r^6h^2}$,
then there exists a partition $V(G)=V_1\cup \dots \cup V_r$ such that $\sum_{i=1}^re(G[V_i])\leq \epsilon n^2$ and $\sum_{1\leq i<j\leq r}e(V_i,V_j)$ attains the maximum.
Moreover,
the following results hold:
\begin{enumerate}[label=(\arabic*)]
\item\label{item:first} For $1\leq i\leq r$,
$\frac{n}{r}-2\sqrt{\epsilon} n<|V_i|<\frac{n}{r}+2\sqrt{\epsilon} n$.
\item\label{item:second} Let $L_0=\{v\in V(G)\ |\ d_G(v)\leq\left(1-\frac{1}{r}-6\sqrt{\epsilon}\right)n\}$.
Then $|L_0|\leq \sqrt{\epsilon} n$.
\item\label{item:third} Let $W_0=\bigcup_{i=1}^r\{v\in V_i\ |\ d_{V_i}(v)\geq \frac{2r}{r-1}\sqrt{\epsilon} n\}$.
Then $|W_0|\leq \frac{r-1}{r}\sqrt{\epsilon} n$.
\end{enumerate}
\end{lemma}


Given a sufficiently small constant $\xi>0$,
Let $\epsilon=\xi^{6h}<\frac{1}{64r^6h^2}$. 
There exists a sufficiently large $n$ such that $\frac{1}{n} \ll \xi^{6h}$.
Since $\xi$ is sufficiently small,
we have $6\sqrt{\epsilon}\le \xi$, $\sqrt{\epsilon}\leq\xi^2$, $\frac{r}{r-1}\sqrt{\epsilon}\le \xi^3.$
Thus, 
for the graph $G^*$, 
there exists a partition $V(G^*)=V_1\cup\dots\cup V_r$
such that $\sum_{i=1}^r e(G^*[V_i])\le \epsilon n^2$ and $\sum_{1\le i<j\le r} e(V_i,V_j)$ attains the maximum.
Moreover, 
\begin{enumerate}[label=(\arabic*)]
\item For each $1\le i\le r$,
$\frac nr-2\xi^2\,n<|V_i|<\frac nr+2\xi^2\,n.$
\item Let $L=\left\{v\in V(G^*): d_{G^*}(v)\le\left(1-\frac1r-\xi\right)n\right\}.$
Then
$|L|\leq \xi^2 n$.
\item Let $W=\bigcup_{i=1}^r \left\{v\in V_i: d_{V_i}(v)\ge 2\xi^3 n\right\}.$
Then $|W|\leq \xi^3 n.$
\end{enumerate}

\begin{lemma}\label{344}
For any $S\subseteq V(G^*)$ with $|S|\leq(t-1)h$,
if there exists an edge $e$ in $G^*[V_i\setminus (L\cup W\cup S)]$ for $i\in \{1,2,\dots,r\}$,
then $G^*-(L\cup W\cup S)$ contains a $K^+(n_1,n_2,\dots,n_r)$ containing the edge $e$.
\end{lemma}

\begin{proof}
Without loss of generality,
assume that $uv$ is an edge within $G^*[V_1\setminus(L\cup W\cup S)]$.
For each $w\in V_i\setminus (L\cup W\cup S),i\in\{1,2,\dots,r\}$,
by the definitions of $L$ and $W$,
since $w\notin L$,
we have $d_{G^*}(w)>\left(1-\frac{1}{r}-\xi\right)n,$
since $w\notin W$,
we have $d_{V_i}(w)<2\xi^3 n.$
Then for $i\neq i'$, we have 
\begin{equation}\label{du}
\begin{split}
d_{V_{i'}}(w) &\geq d_{G^*}(w)-d_{V_i}(w)-\sum_{j=1,j\neq i,i'}^r|V_j| \\
& > \left(1-\frac{1}{r}-\xi\right)n-2\xi^3 n-(r-2)\left(\frac{n}{r}+2\xi^2 n\right) \\
& > \left(\frac{1}{r}-\xi-2\xi^2 (r-2) -2\xi^3 \right)n.
\end{split}
\end{equation}
For the number of common neighbors of $u$ and $v$ in $V_2\setminus (L\cup W\cup S)$,
by Lemma \ref{Cioaba} and Eq. (\ref{du}),
we have
\begin{align*}
&\left|N_{V_2}(u)\cap N_{V_2}(v)\setminus (L\cup W\cup S)\right|\\
\geq& d_{V_2}(u)+d_{V_2}(v)-\left|N_{V_2}(u)\cup N_{V_2}(v)\right|-|L|-|W|-|S|\\
>& 2\left(\frac{n}{r}-2\xi^2 n(r-2)-\xi n-2\xi^3 n\right)-\left(\frac{n}{r}+2\xi^2 n\right)-\xi^2 n-\xi^3 n-(t-1)h\\
=&\left(\frac{1}{r}-2\xi -(4r-5)\xi^2 -5\xi^3 \right)n-(t-1)h\\
>&n_2.
\end{align*}
Therefore, 
$u$ and $v$ have at least $n_2$ common neighbors in $V_2\setminus (L\cup W\cup S)$,
denoted by $u^{(2)}_{1},u^{(2)}_{2},\dots,u^{(2)}_{n_2}$.
Consequently,
the subgraph induced by $\{u,v\}\cup \{u^{(2)}_{1},u^{(2)}_{2},\dots,u^{(2)}_{n_2}\}$ contains a copy of $K^+(2,n_2)$.
Hence,
for $\ell\in \{2,\dots,r-1\}$,
assume that there exist at least $n_{\ell}$ vertices in $V_{\ell}\setminus (L\cup W\cup S)$ such that the subgraph induced by $\{u,v\}\cup \{u^{(2)}_{1},u^{(2)}_{2},\dots,u^{(2)}_{n_2}\}\cup\dots \cup \{u^{(\ell)}_{1},u^{(\ell)}_{2},\dots,u^{(\ell)}_{n_{\ell}}\}$ contains a copy of $K^+(2,n_2,\dots,n_{\ell})$.
Next,
we consider the common neighbors of $2+\sum_{i=2}^{\ell}n_i$ vertices in $V_{\ell+1}\setminus (L\cup W\cup S)$.
For any $w\in U=\{u,v,u^{(2)}_{1},\dots,u^{(\ell)}_{n_{\ell}}\}$,
by Lemma \ref{Cioaba} and Eq. (\ref{du}) we have
\begin{align*}
&\left|\bigcap_{w\in U}N_{V_{\ell+1}}(w)\setminus (L\cup W\cup S)\right|\\
\intertext{\vspace{-\baselineskip}\pagebreak}
\geq& \sum_{w\in U}d_{V_{\ell+1}}(w)-(|U|-1)\left|\bigcup_{w\in U}N_{V_{\ell+1}}(w)\right|-|L|-|W|-|S|\\
\geq& \Big(2+\sum_{i=2}^{\ell}n_i\Big)\left(\frac{n}{r}-2\xi^2 n(r-2)-\xi n-2\xi^3 n\right)-\Big(1+\sum_{i=2}^{\ell}n_i\Big)\left(\frac{n}{r}+2\xi^2 n\right)-\xi^2 n-\xi^3 n\\&-(t-1)h\\
=&\left(\frac{1}{r}-\Big(2+\sum_{i=2}^{\ell}n_i\Big)\xi-\Big(4r-5+2(r-1)\sum_{i=2}^{\ell}n_i\Big)\xi^2- \Big(5+2\sum_{i=2}^{\ell}n_i\Big)\xi^3\right)n-(t-1)h\\
>&n_{\ell+1}.
\end{align*}
Therefore,
we can choose $n_{\ell+1}$ vertices in $V_{\ell+1}\setminus (L\cup W\cup S)$ which are adjacent to all vertices in $K^+(2,n_2,\dots,n_{\ell})$,
the subgraph induced by these vertices contains a copy of $K^+(2,n_2,\dots,n_{\ell+1})$.
Repeating the above operation,
we obtain a copy of $K^+(2,n_2,\dots,n_r)$ in $G^*-(L\cup W\cup S)$.
Next,
we consider the common neighbors of $U'=\{u^{(2)}_{1},\dots,u^{(r)}_{n_r}\}$ in $V_1\setminus (W\cup L\cup S \cup \{u,v\})$.
By Lemma \ref{Cioaba} and Eq. (\ref{du}),
we have
\begin{align*}
&\left|\bigcap_{w\in U'}N_{V_1}(w)\setminus (L\cup W\cup S \cup\{u,v\})\right|\\
\geq& \sum_{w\in U'}d_{V_1}(w)-(|U'|-1)\left|\bigcup_{w\in U'}N_{V_1}(w)\right|-|W|-|L|-|S|-2\\
\geq& \Big(\sum_{i=2}^rn_i\Big)\left(\frac{n}{r}-2\xi^2 n(r-2)-\xi n-2\xi^3 n\right)-\Big(\sum_{i=2}^rn_i-1\Big)\left(\frac{n}{r}+2\xi^2 n\right)-\xi^2 n-\xi^3 n\\&-(t-1)h-2\\
>&\left(\frac{1}{r}-\Big(\sum_{i=2}^rn_i\Big)\xi-\Big(2(r-1)\sum_{i=2}^rn_i-1\Big)\xi^2-\Big(2\sum_{i=2}^rn_i+1\Big)\xi^3 \right)n-(t-1)h-2\\
>&n_1-2.
\end{align*}
Then we can find $n_1-2$ vertices that are common neighbors of the vertices in $U'$.
Thus,
we obtain a copy of $K^+(n_1,n_2,\dots,n_r)$ in $G^*-(L\cup W\cup S)$.
Hence,
$G^*-(L\cup W\cup S)$ contains a $K^+(n_1,n_2,\dots,n_r)$ containing $uv$.
\end{proof}

\begin{lemma}\label{3441}
Let $H=\bigcup_{i=1}^r G^*[V_i\setminus(L\cup W)].$
Then the matching number $\mu(H)\leq t-1$.
Moreover,
for each $i\in\{1,2,\dots,r\}$,
there exists an independent set $I_i\subseteq V_i\setminus(L\cup W)$ such that $|I_i|\geq |V_i\setminus(L\cup W)|-2(t-1).$
\end{lemma}

\begin{proof}
Suppose to the contrary that $\mu(H)\ge t$.
Then there exist $t$ disjoint edges $e_1,\dots,e_t \in E(H)$.
For some $1\le q\le t$,
suppose that $F_1,\dots,F_{q-1}$ have already been constructed,
where each $F_i \ (i=1,\dots,q-1)$ is a copy of $K^+(n_1,\dots,n_r)$,
and these copies are disjoint.
Let $S_q=(\bigcup_{i=1}^{q-1}V(F_i))\cup (\bigcup_{i=q+1}^t V(e_i))$.
Then $|S_q|\le (q-1)h+2(t-q)\le (t-1)h.$
Moreover,
since the edges $e_1,\dots,e_t$ are disjoint and the vertices of $e_{q+1},\dots,e_t$ have been included in $S_q$,
we have $e_q$ is contained in $G^*[V_i\setminus(L\cup W\cup S_q)].$
By Lemma \ref{344}, 
we know that the graph $G^*-(L\cup W\cup S_q)$ contains a copy of $K^+(n_1,\dots,n_r)$ containing the edge $e_q$,
denoted by $F_q$.
Obviously,
$F_q$ is disjoint from $F_1,\dots,F_{q-1}$ and also disjoint from the edges $e_{q+1},\dots,e_t$.
Repeating this process for $q=1,\dots,t$,
we obtain $t$ disjoint copies of $K^+(n_1,\dots,n_r)$ in $G^*$,
a contradiction.
Hence $\mu(H)\le t-1$.

Let $M_i$ be the maximum matching in $G^*[V_i\setminus (L\cup W)]$.
Since $G^*[V_i\setminus (L\cup W)]\subseteq H$, 
we have $|M_i|\le \mu(H)\le t-1.$
Let $I_i=(V_i\setminus (L\cup W))\setminus V(M_i).$
Since $M_i$ is a maximum matching, 
we have $I_i$ is independent. 
Otherwise, 
if there exists an edge inside $I_i$, 
then this edge would be disjoint from all edges of $M_i$ and hence could be added to $M_i$, 
contradicting the maximality of $M_i$. 
Therefore,
$|I_i|=|V_i\setminus (L\cup W)|-|V(M_i)|=|V_i\setminus (L\cup W)|-2|M_i|\geq |V_i\setminus (L\cup W)|-2(t-1).$
\end{proof}

Obviously, 
$\mathcal{A}_s(G^*)$ is a nonnegative tensor.
By Lemma \ref{Lemma2}, 
$\rho_s(G^*)$ is an eigenvalue of $\mathcal{A}_s(G^*)$,
with a nonnegative eigenvector corresponding to it.
Let $x=(x_1,x_2,\dots,x_n)^{\top}$ be a nonnegative eigenvector corresponding to $\rho_s(G^*)$ with $\max_{v\in V(G^*)}x_v=1.$
In the following three lemmas, 
we shall give some local structural properties of $G^*$.

\begin{lemma}\label{377}
Let $v_0$ be the vertex such that $x_{v_0}=\max\{x_v\ |\ v\in V(G^*)\setminus W \}$.
Then $$x_{v_0}^{s-1}\geq \frac{(s-1)!}{4}{r-1\choose s-1}\left(\frac{1}{r}\right)^{s-1}.$$
\end{lemma}

\begin{proof}
Let $u_0$ be the vertex with $x_{u_0}=\max\{x_v\ |\ v\in V(G^*)\}=1$.
We have
\begin{align*}
\rho_s(G^*)=\rho_s(G^*)x_{u_0}^{s-1}&=\sum_{\{u_0,i_2,\dots,i_s\}\in C_s(G^*)}x_{i_2}\dots x_{i_s}\\
&=\sum_{\substack{\{u_0,i_2,\dots,i_s\}\in C_s(G^*)\\|\{i_2,\dots,i_s\}\cap W|=0}}x_{i_2}\dots x_{i_s}
+\sum_{j=1}^{s-1}\sum_{\substack{\{u_0,i_2,\dots,i_s\}\in C_s(G^*)\\|\{i_2,\dots,i_s\}\cap W|=j}}x_{i_2}\dots x_{i_s}\\
&\leq {n-1\choose s-1}x_{v_0}^{s-1}+\sum_{j=1}^{s-1}{|W|\choose j}{n-|W|\choose s-1-j}\\
\intertext{\vspace{-\baselineskip}\pagebreak}
&\leq\frac{n^{s-1}}{(s-1)!}x_{v_0}^{s-1}+ \sum_{j=1}^{s-1}\frac{|W|^j}{j!}\frac{(n-|W|)^{s-1-j}}{(s-1-j)!}.
\end{align*}
Since $|W|\leq \xi^3 n$,
we have 
\begin{align*}
\sum_{j=1}^{s-1}\frac{|W|^j}{j!}\frac{(n-|W|)^{s-1-j}}{(s-1-j)!}&=
\sum_{j=1}^{s-1}\frac{(\xi^3)^jn^{s-1}}{j!(s-1-j)!}=\frac{n^{s-1}}{(s-1)!}\sum_{j=1}^{s-1}{s-1\choose j}(\xi^3)^j\\
&\leq\frac{n^{s-1}}{(s-1)!}\big((1+\xi^3)^{s-1}-1\big).
\end{align*}
Since $\xi$ is sufficiently small,
we have $(1+\xi^3)^{s-1}-1\leq \frac{3(s-1)!}{4}{r-1\choose s-1}(\frac{1}{r})^{s-1}$.
By Lemma \ref{311}, 
we have $\rho_s(G^*)\geq \binom{r-1}{s-1}\left(\frac nr\right)^{s-1}$.
Hence,
we have 
$x_{v_0}^{s-1}\geq \frac{(s-1)!}{4}{r-1\choose s-1}\left(\frac{1}{r}\right)^{s-1}.$
\end{proof}

\begin{lemma}\label{388}
Let $v_0$ be the vertex such that $x_{v_0}=\max\{x_v\ |\ v\in V(G^*)\setminus W \}$.
Then $v_0\notin L$.
\end{lemma}

\begin{proof}
Suppose to the contrary that $v_0\in L$.
Without loss of generality,
assume that \(v_0\in V_1\).
Let $C_s(v_0)$ be the set of $s$-cliques containing $v_0$ in $G^*$.
$C_s(v_0)$ can be partitioned into the following four classes:
\begin{align*}
C_1&=\{c_s\ |\ (c_s\setminus \{v_0\})\cap (W\cup L)\neq \emptyset\};\\
C_2&=\{c_s\ |\ (c_s\setminus \{v_0\})\cap (W\cup L)=\emptyset,(c_s\setminus \{v_0\})\cap V_1\neq\emptyset\};\\
C_3&=\{c_s\ |\ (c_s\setminus \{v_0\})\subseteq V(G^*)\setminus (V_1\cup W\cup L),(c_s\setminus \{v_0\})\cap (\cup_{i=2}^r(V_i\setminus  I_i))\neq\emptyset\};\\
C_4&=\{c_s\ |\ (c_s\setminus \{v_0\})\subseteq \cup_{i=2}^rI_i\}.
\end{align*}
Let $S_i=\sum_{c_s\in C_i}\prod_{v\in c_s\setminus \{v_0\}}x_v$,
We have $$\rho_s(G^*)x_{v_0}^{s-1}=\sum_{\{v_0,i_2,\dots,i_s\}\in C_s(v_0)}x_{i_2}\dots x_{i_s}=S_1+S_2+S_3+S_4.$$
Recall that $\max_{v\in V(G^*)}x_v=1$.
Since $|W\cup L|\leq(\xi^2+\xi^3) n$,
we have $$S_1=\sum_{c_s\in C_1}\prod_{v\in c_s\setminus \{v_0\}}x_v\leq|C_1|\leq |W\cup L|n^{s-2}\leq (\xi^2+\xi^3) n^{s-1}.$$
Since $v_0\notin W$,
we have $d_{V_1}(v_0)<2\xi^3 n.$
Then 
$$S_2=\sum_{c_s\in C_2}\prod_{v\in c_s\setminus \{v_0\}}x_v\leq |C_2|\leq d_{V_1}(v_0)n^{s-2}<2\xi^3 n^{s-1}.$$
By Lemma \ref{3441}, 
we have $|V_i\setminus(W\cup L)|-|I_i|\leq 2(t-1)$ for $2\leq i\leq r$.
Then $$S_3=\sum_{c_s\in C_3}\prod_{v\in c_s\setminus \{v_0\}}x_v\leq |C_3|\leq 2(t-1)(r-1)n^{s-2}.$$
It remains to estimate $S_4$.
Let $H_0=G^*[N(v_0)\cap(I_2\cup\dots\cup I_r)]$.
We have $|C_4|=|C_{s-1}(H_0)|$.
Since $v_0\in L$ and $I_i\subseteq V_i$ for each $i\in \{1,2,\dots,r\}$,
we have $|V(H_0)|\leq \sum_{j=2}^rd_{V_j}(v_0)\leq d_{G^*}(v_0)\leq \left(1-\frac{1}{r}-\xi\right)n.$
Then 
\begin{align*}
|C_4|&\leq {r-1\choose s-1}\left(\frac{|V(H_0)|}{r-1}\right)^{s-1}\\
&\leq {r-1\choose s-1}\left(1-\frac{r\xi}{r-1}\right)^{s-1}\left(\frac{n}{r}\right)^{s-1}\\
&\leq {r-1\choose s-1}\left(\frac{n}{r}\right)^{s-1}-\frac{r\xi}{r-1}{r-1\choose s-1}\left(\frac{n}{r}\right)^{s-1},
\end{align*}
where the last inequality holds since $\left(1-\frac{r\xi}{r-1}\right)^{s-1}\leq 1-\frac{r\xi}{r-1}$.
Thus,
$$S_4=\sum_{c_s\in C_4}\prod_{v\in c_s\setminus \{v_0\}}x_v\leq |C_4|x_{v_0}^{s-1}\leq {r-1\choose s-1}\left(\frac{n}{r}\right)^{s-1}x_{v_0}^{s-1}-\frac{r\xi}{r-1}{r-1\choose s-1}\left(\frac{n}{r}\right)^{s-1}x_{v_0}^{s-1}.$$
Combine $S_1$,
$S_2$,
$S_3$ and $S_4$,
we have
\begin{align*}
 \rho_s(G^*)x_{v_0}^{s-1}=&S_1+S_2+S_3+S_4\\
 \leq& {r-1\choose s-1}\left(\frac{n}{r}\right)^{s-1}x_{v_0}^{s-1}-\frac{r\xi}{r-1}{r-1\choose s-1}\left(\frac{x_{v_0}}{r}\right)^{s-1}n^{s-1}+(\xi^2+3\xi^3)n^{s-1}\\&+2(t-1)(r-1)n^{s-2}\\
 <&{r-1\choose s-1}\left(\frac{n}{r}\right)^{s-1}x_{v_0}^{s-1},
\end{align*}
where the last inequality holds since $\xi$ is sufficiently small, 
$n$ is sufficiently large and $x_{v_0}^{s-1}\geq \frac{(s-1)!}{4}{r-1\choose s-1}(\frac{1}{r})^{s-1}$.
Therefore,
$\rho_s(G^*)<{r-1\choose s-1}\left(\frac{n}{r}\right)^{s-1},$
a contradiction to Lemma \ref{311}.
Hence,
$v_0\notin L$.
\end{proof}

\begin{lemma}\label{3991}
$L\subseteq W$.
\end{lemma}

\begin{proof}
Suppose to the contrary that $u\in L\setminus W$.
Recall that $v_0$ is the vertex such that $x_{v_0}=\max\{x_v\ |\ v\in V(G^*)\setminus W\}.$
Without loss of generality,
assume that $v_0\in V_1$ and $u\in V_k$,
where $k\in \{1,2,\dots,r\}$.
Construct a graph $G'$ from $G^*$ by deleting all edges incident with $u$,
and then joining $u$ to every vertex in $\bigcup_{2\leq i\leq r} I_i$.
We first prove that $\rho_s(G')>\rho_s(G^*)$.
It is known from the structure of $G'$ that
\begin{equation}
\begin{aligned}\label{000}
\rho_s(G')-\rho_s(G^*)&\geq \frac{x^\top\mathcal{A}_s(G')x^{s-1}}{\|x\|_s^s}-\frac{x^\top\mathcal{A}_s(G^*)x^{s-1}}{\|x\|_s^s}\\
&=\frac{sx_u\big((\mathcal{A}_s(G')x^{s-1})_u-(\mathcal{A}_s(G^*)x^{s-1})_u\big)}{\|x\|_s^s}\\
&=\frac{sx_u\big((\mathcal{A}_s(G')x^{s-1})_u-\rho_s(G^*)x_u^{s-1}\big)}{\|x\|_s^s}.
\end{aligned}
\end{equation}
Recall that $S_4$ is the contribution of those $s$-cliques $c_s$ in $G^*$ such that $c_s\setminus\{v_0\}\subseteq \bigcup_{i=2}^rI_i.$
The estimates for $S_1$, 
$S_2$ and $S_3$ are shown in Lemma \ref{388},
we have
$$S_4=\rho_s(G^*)x_{v_0}^{s-1}-S_1-S_2-S_3>\rho_s(G^*)x_{v_0}^{s-1}-(\xi^2+3\xi^3)n^{s-1}-2(t-1)(r-1)n^{s-2}.$$
For each $c_s\in C_4$ in $G^*$, 
by the construction of $G'$, 
$\{u\}\cup(c_s\setminus\{v_0\})$ forms an $s$-clique in $G'$. 
Hence,
\begin{equation}\label{4}
(\mathcal A_s(G')x^{s-1})_u\geq S_4>\rho_s(G^*)x_{v_0}^{s-1}-(\xi^2+3\xi^3)n^{s-1}-2(t-1)(r-1)n^{s-2}.
\end{equation}
Since $u\in L$, 
$u\notin W$,
similar to the proof in Lemma \ref{388},
we have 
\begin{equation}\label{5} 
\begin{aligned}
\rho_s(G^*)x_u^{s-1}\leq&\binom{r-1}{s-1}\left(\frac{n}{r}\right)^{s-1}x_{v_0}^{s-1}-
\frac{r\xi}{r-1}{r-1\choose s-1}\left(\frac{x_{v_0}}{r}\right)^{s-1}n^{s-1}\\
&+(\xi^2+3\xi^3)n^{s-1}+
2(t-1)(r-1)n^{s-2}.  
\end{aligned}
\end{equation}
Combining Eqs. (\ref{4}) and (\ref{5}),
we obtain
\begin{align*}
(\mathcal A_s(G')x^{s-1})_u- \rho_s(G^*)x_u^{s-1}
>&\left(\rho_s(G^*)-\binom{r-1}{s-1}\left(\frac nr\right)^{s-1}\right)x_{v_0}^{s-1}+\frac{r\xi}{r-1}{r-1\choose s-1}\left(\frac{x_{v_0}}{r}\right)^{s-1}n^{s-1}\\
&-2(\xi^2+3\xi^3)n^{s-1}-4(t-1)(r-1)n^{s-2}.
\end{align*}
Since $\xi$ is sufficiently small and $n$ is sufficiently large,
by Lemmas \ref{311} and \ref{377}, 
it follows that $(\mathcal A_s(G')x^{s-1})_u-(\mathcal{A}_s(G^*)x^{s-1})_u>0$.
If $x_u>0$, 
then by Eq. (\ref{000}) we have $\rho_s(G')-\rho_s(G^*)>0$.
It remains to consider the case $x_u=0$.
By Eq. (\ref{4}), Lemmas \ref{311} and \ref{377},
for sufficiently small $\xi$ and sufficiently large $n$,
we have $$(\mathcal A_s(G')x^{s-1})_u\geq S_4\geq \frac12\rho_s(G^*)x_{v_0}^{s-1}.$$
Let $y=x+\eta x_{v_0}\mathbf{e}_u,\eta=\frac{1}{2}.$
Since $x_u=0$,
we have $\|y\|_s^s=\|x\|_s^s+\eta^s x_{v_0}^s.$
Furthermore,
$G'$ and $G^*$ differ only in edges incident with $u$.
Hence
\begin{align*}
s\sum_{c_s\in C_s(G')}\prod_{v\in c_s}y_v&\geq s\sum_{c_s\in C_s(G^*)}\prod_{v\in c_s}x_v+s\eta x_{v_0}(\mathcal A_s(G')x^{s-1})_u \\
&\geq\rho_s(G^*)\|x\|_s^s+\frac{s}{2}\eta\rho_s(G^*)x_{v_0}^s \\
&>\rho_s(G^*)\left(\|x\|_s^s+\eta^s x_{v_0}^s\right) \\
&=\rho_s(G^*)\|y\|_s^s,
\end{align*}
where the last second inequality follows from $\eta=1/2$ and $s\geq2$.
Consequently,
by Lemma \ref{Lemma1}, 
we have $$\rho_s(G')\geq\frac{s\sum_{c_s\in C_s(G')}\prod_{v\in c_s}y_v}{\|y\|_s^s}>\rho_s(G^*).$$

Next,
we prove that $G'$ is $tK^+(n_1,n_2,\dots,n_r)$-free.
Otherwise,
$G'$ contains a copy of $tK^+(n_1,n_2,\dots,n_r)$,
denoted by $F$. 
Since $G^*$ is $tK^+(n_1,n_2,\dots,n_r)$-free and $G'$ differs from $G^*$ only in edges incident with $u$, 
there exists a copy of $K^+(n_1,n_2,\dots,n_r)$ in $F$ containing $u$ and we denote it by $F_1$. 
Let $F'=F-V(F_1).$
Then $F'$ is the union of the remaining $t-1$ copies of $K^+(n_1,n_2,\dots,n_r)$, 
and hence $|V(F')|=(t-1)h.$
Let $N_{F_1}(u)=\{u_1,u_2,\dots,u_m\}.$
By the construction of $G'$,
$N_{F_1}(u)\subseteq \bigcup_{2\leq i\leq r} I_i$. 
For any $u'\in N_{F_1}(u)\cap I_i$, 
where $i\neq 1$, 
we have $u'\notin L\cup W$. 
Therefore $d_{G^*}(u')>\left(1-\frac1r-\xi\right)n$ and $d_{V_i}(u')<2\xi^3 n.$
It follows that
\begin{align*}
d_{V_1}(u')&\geq d_{G^*}(u')-d_{V_i}(u')-\sum_{2\leq j\leq r, j\neq i} |V_j|  \\
&>\left(1-\frac1r-\xi\right)n-2\xi^3 n-(r-2)\left(\frac nr+2\xi^2\,n\right)  \\
&>\left(\frac {1}{r}-\xi-2(r-2)\xi^2-2\xi^3 \right)n .
\end{align*}
Thus, by Lemma \ref{Cioaba},
we have
\begin{align*}
\left|\bigcap_{\ell=1}^m N_{V_1}(u_\ell)\right|
&\geq \sum_{\ell=1}^m d_{V_1}(u_\ell)-(m-1)|V_1| \\
&>m\left(\frac nr-\xi n-2\xi^3 n-2(r-2)\xi^2\,n\right)-(m-1)\left(\frac nr+2\xi^2\,n\right) \\
&>\left(\frac {1}{r}-m\xi-2m(r-1)\xi^2-2m\xi^3 \right)n\\
&>th.
\end{align*}
Hence there exists a vertex $u''\in\left(\bigcap_{\ell=1}^m N_{V_1}(u_\ell)\right)\setminus V(F).$
Replacing $u$ by $u''$ in \(F_1\), 
we obtain a copy of $K^+(n_1,n_2,\dots,n_r)$ in $G^*$,
denoted by $F_1'$. 
Moreover, 
$F_1'$ is disjoint from $F'$. 
Hence $F_1'\cup F'$ is a copy of $tK^+(n_1,n_2,\dots,n_r)$ in $G^*$, 
a contradiction. 
Therefore, 
$G'$ is $tK^+(n_1,n_2,\dots,n_r)$-free. On the other hand, We have shown $\rho_s(G')>\rho_s(G^*)$, which contradicts the extremality of $G^*$  among all $tK^+(n_1,n_2,\dots,n_r)$-free graphs on $n$ vertices.
Hence,
$L\subseteq W.$
\end{proof}

\begin{lemma}\label{3112}
For any vertex set $S\subseteq V(G^*)$ with $|S|\leq th$,
and for each $w\in W$,
$G^*-\big((W\cup S)\setminus \{w\}\big)$ contains a copy of $K^+(n_1,n_2,\dots,n_r)$ containing $w$.
\end{lemma}

\begin{proof}
Without loss of generality,
assume that $w\in V_1$.
Since the partition $V(G^*)=V_1\cup\dots\cup V_r$ attains the maximum $\sum_{1\le i<j\le r}e(V_i,V_j)$,
moving $w$ from $V_1$ to $V_j$ cannot increase this sum.
Hence $d_{V_j}(w)\geq d_{V_1}(w)\geq 2\xi^3 n$ for $j=1,\dots,r.$
Let $X=(W\cup S)\setminus \{w\}$ and $A_j=N_{V_j}(w)\setminus X$ for $j=1,\dots,r$.
Since $|W|\le \xi^3 n$ and $|S|\le th$,
for sufficiently large $n$,
we have $|A_j|\geq 2\xi^3 n-|W|-|S|>\frac{\xi^3 n}{2}.$

Let $M=\sum_{1\leq i<j\le r}\overline e(V_i,V_j).$
Since $G^*$ can be obtained from $T_r(n)$ by adding and deleting at most $\epsilon n^2$ edges,
we have $e(G^*)\geq e(T_r(n))-\epsilon n^2.$
Moreover,
$\sum_{1\le i<j\le r}|V_i||V_j|\leq e(T_r(n))$ and $\sum_{i=1}^r e(G^*[V_i])\leq \epsilon n^2.$
Thus
\begin{align*}
M&=\sum_{1\leq i<j\leq r}\big(|V_i||V_j|-e(V_i,V_j)\big)\\
&=\sum_{1\le i<j\le r}|V_i||V_j|+\sum_{i=1}^r e(G^*[V_i])-e(G^*)\\
&\leq e(T_r(n))+\epsilon n^2-\big(e(T_r(n))-\epsilon n^2\big)\\
&=2\epsilon n^2.
\end{align*}
Consequently,
$\sum_{j=2}^r \overline e(A_1,A_j)\le M\leq 2\epsilon n^2.$

For each $u\in A_1$,
define $m(u)=\sum_{j=2}^r |A_j\setminus N(u)|.$
Then $\sum_{u\in A_1}m(u)=\sum_{j=2}^r \overline e(A_1,A_j)\leq 2\epsilon n^2.$
Hence there exists $z\in A_1$ such that
$m(z)\leq \frac{2\epsilon n^2}{|A_1|}<\frac{4\epsilon n}{\xi^3}.$
That is,
$\sum_{j=2}^r |A_j\setminus N(z)|\leq\frac{4\epsilon n}{\xi^3}.$
For $j=2,\dots,r$,
denote $C_j=A_j\cap N(z).$
Then
\begin{align*}
|C_j|&=|A_j|-|A_j\setminus N(z)|\geq|A_j|-\sum_{i=2}^r |A_i\setminus N(z)|>\frac{\xi^3 n}{2}-\frac{4\epsilon n}{\xi^3}=\left(\frac{\xi^3 }{2}-4\xi^{6h-3}\right)n>\frac{\xi^3 n}{4},
\end{align*}
where the last inequality follows from $\epsilon= \xi^{6h}$ and $\xi$ is sufficiently small .

We now choose sets $U_j\subseteq C_j,|U_j|=n_j, j=2,\dots,r$ such that every vertex of $U_i$ is adjacent to every vertex of $U_j$ for all $2\leq i<j\leq r$.
Indeed,
since $|C_j|>\xi^3 n/4$ and $n_1,\dots,n_r$ are fixed,
the total number of choices of $U_2,\dots,U_r$ is at least $c(\xi^3)^{h-n_1}n^{h-n_1},$
where $c>0$ depends only on $n_1,\dots,n_r$.
On the other hand,
the number of choices for which there exist $2\le i<j\le r,u\in U_i$,
and $v\in U_j$ such that $uv\notin E(G^*)$ is at most $\sum_{2\le i<j\le r}\overline e(C_i,C_j)\cdot C_0 n^{h-n_1-2}.$
Since $\sum_{2\le i<j\le r}\overline e(C_i,C_j)\le M\le 2\epsilon n^2,$
this number is at most $C\epsilon n^{h-n_1},$
where $C_0,C$ are positive constants. 
Since $\epsilon\le \xi^{6h}$ and $\xi$ is sufficiently small,
we have $C\epsilon = C\xi^{6h}<c(\xi^3)^{h-n_1}$.
Therefore there exists a choice of $U_2,\dots,U_r$ such that every vertex of
$U_i$ is adjacent to every vertex of $U_j$ whenever $2\le i<j\le r$.

Let $U=\bigcup_{j=2}^r U_j$ and $m=|U|=n_2+\dots+n_r=h-n_1.$
Since \(U_j\subseteq C_j\subseteq A_j\cap N(z)\), 
both \(w\) and \(z\) are adjacent to every vertex of \(U\). 
Moreover, 
the sets \(U_2,\dots,U_r\) are pairwise completely joined.
It remains to choose $n_1-2$ additional vertices from $V_1$.
Since $U\cap W=\emptyset$ and $L\setminus W=\emptyset$,
we have $U\cap L=\emptyset$.
Thus,
for every $y\in U$,
if $y\in V_j$ with $j\in\{2,\dots,r\}$,
then $d_{G^*}(y)>\left(1-\frac1r-\xi\right)n.$
Since $y\notin W$,
we have $d_{V_j}(y)<2\xi^3 n.$
Since $|V_\ell|\le \frac nr+2\xi^2\,n,$
we have
\begin{equation}\label{7}
\begin{split}
d_{V_1}(y)&\geq d_{G^*}(y)-d_{V_j}(y)-\sum_{\ell\neq 1,j}|V_\ell|\\
&>\left(1-\frac1r-\xi\right)n-2\xi^3 n-(r-2)\left(\frac {n}{r}+2\xi^2\,n\right)\\
&=\left(\frac 1r-\xi-2(r-2)\xi^2 -2\xi^3 \right)n.
\end{split}
\end{equation}
For the number of common neighbors of $U$ in $V_1\setminus (X\cup\{w,z\})$,
by Lemma \ref{Cioaba} and Eq. (\ref{7}),
we have
\begin{align*}
\left|\bigcap_{y\in U}N_{V_1}(y)\setminus (X\cup\{w,z\})\right|&\geq\sum_{y\in U}d_{V_1}(y)-(m-1)|V_1|-|X|-2\\
&>\frac nr-\big(m\xi+(2m+1)\xi^3+(2m(r-1)-2)\xi^2\big)n-th-2\\
&\geq\left(\frac {1}{r}-h\xi-2h(r-1)\xi^2-(2h+1)\xi^3\right)n-th-2\\
&>n_1-2.
\end{align*}
Hence we may choose a set $U_1\subseteq\bigcap_{y\in U}N_{V_1}(y)\setminus (X\cup\{w,z\})$ with $|U_1|=n_1-2.$

Now consider the vertex set $\{w,z\}\cup U_1\cup (\bigcup_{j=2}^r U_j).$
Since $z\in A_1\subseteq N_{V_1}(w)$,
we have $wz\in E(G^*)$.
Both $w$ and $z$ are adjacent to every vertex of $U$,
every vertex of $U_1$ is adjacent to every vertex of $U$,
and the sets $U_2,\dots,U_r$ are pairwise completely joined.
Therefore this vertex set contains a copy of $K^+(n_1,n_2,\dots,n_r)$,
where $wz$ is the extra internal edge in the first part.
All chosen vertices are not contained in $X$.
Hence we obtain a copy of $K^+(n_1,n_2,\dots,n_r)$ in $G^*-\big((W\cup S)\setminus \{w\}\big)$.
\end{proof}

\begin{lemma}\label{312}
Let $H=\bigcup_{i=1}^rG^*[V_i\setminus W]$.
Then the matching number $\mu(H)\leq t-1-|W|$.
\end{lemma}

\begin{proof}
Firstly,
we claim that $|W|\leq t-1$.
Otherwise,
$|W|\geq t$,
then there exist at least $t$ vertices $w_1,\dots,w_t\in W$.
By Lemma \ref{3112},
we obtain $t$ vertex-disjoint copies of $K^+(n_1,n_2,\dots,n_r)$ in $G^*$,
a contradiction.

Suppose to the contrary that $\mu(H)\geq t-|W|$.
Then there exists $t-|W|$ disjoint edges $e_1,\dots,e_{t-|W|}\in E(H)$.
We first construct $t-|W|$ disjoint copies of $K^+(n_1,n_2,\dots,n_r)$,
denoted by $F_1,\dots,F_{t-|W|}$,
all disjoint from $W$,
such that each $F_i$ contains the edge $e_i$,
where $i\in \{1,\dots,t-|W|\}$.
By induction on $t-|W|$,
we assume that $F_1,\dots,F_{i-1}$ have been chosen.
Let $X_i= \left(\bigcup_{j=1}^{i-1}V(F_j)\right)\cup \left(\bigcup_{\ell=i+1}^{t-|W|}V(e_{\ell})\right)$.
Then $|X_i|\leq  (t-1)h.$
Since the edges $e_1,\dots,e_{t-|W|}$ are disjoint,
the edge $e_i$ is disjoint from $W\cup X_i$.
Hence,
by Lemma \ref{344},
we obtain a copy of $K^+(n_1,n_2,\dots,n_r)$ in $G^*-(W\cup X_i)$ containing $e_i$,
denoted by $F_i$.
Thus we obtain $t-|W|$ disjoint copies of $K^+(n_1,n_2,\dots,n_r)$,
denoted by $F_1,\dots,F_{t-|W|}$,
all disjoint from $W$.

Let $W=\{w_1,\dots,w_{|W|}\}$.
Next we construct $|W|$ disjoint copies of $K^+(n_1,n_2,\dots,n_r)$,
one containing each $w_k$.
Suppose that $F_1,\dots,F_{t-|W|+k-1}$ have already been chosen.
Let $Y_k=\bigcup_{j=1}^{t-|W|+k-1}V(F_j)$.
Then $|Y_k|\leq (t-|W|+k-1)h\leq (t-1)h.$
Hence,
by Lemma \ref{3112},
the graph $G^*-\big((W\cup Y_k)\setminus \{w_k\}\big)$ contains a copy of $K^+(n_1,n_2,\dots,n_r)$ containing $w_k$.
Hence,
we obtain $t-|W|+|W|=t$ disjoint copies of $K^+(n_1,n_2,\dots,n_r)$ in $G^*$,
a contradiction.
Therefore,
$\mu(H)\le t-1-|W|.$
\end{proof}

\begin{lemma}\label{3221}
$G^*$ is $tK_{r+1}$-free.
\end{lemma}

\begin{proof}
Suppose to the contrary that $G^*$ contains $t$ disjoint copies of $K_{r+1}$,
denoted by $F_1,F_2,\dots,F_t$.
Since $V(G^*)=V_1\cup\dots\cup V_r$ is an $r$-partition of $V(G^*)$,
each $F_j$ has $r+1$ vertices distributed among the $r$ parts $V_1,\dots,V_r$.
Hence,
there exist two vertices of $F_j$ that belong to the same part for $j\in\{1,2,\dots,t\}$.
Since $F_j$ is a clique,
these two vertices are adjacent.

Since the copies $F_1,F_2,\dots,F_t$ are disjoint,
at most $|W|$ of them have a vertex in $W$.
Therefore at least $t-|W|$ of these copies are disjoint from $W$.
For each copy $F_j$ which is disjoint from $W$,
choose an edge of $F_j$ whose two endpoints belong to the same part.
Then all the chosen edges are contained in $H=\bigcup_{i=1}^r G^*[V_i\setminus W].$
Moreover,
since the copies $F_1,F_2,\dots,F_t$ are disjoint,
these chosen edges are disjoint.
Hence these edges form a matching in $H$ of size at least $t-|W|$.
Consequently,
$\mu(H)\ge t-|W|,$
contradiction with Lemma \ref{312}.
Therefore,
$G^*$ is $tK_{r+1}$-free.
\end{proof}

\begin{lemma}\label{322}
$G^*$ is $s$-clique connected.
\end{lemma}

\begin{proof}
Suppose that $G^*$ is not $s$-clique connected.
Then there exists an $s$-clique connected component $H$ of $G^*$ such that $ \rho_s(G^*)=\rho_s(H)$.
For $j\in V(G^*)\setminus V(H)$,
choose an $s$-clique $c_s=\{u_1,u_2,\dots,u_s\}\in C_s(H)$ such that $|N(j)\cap c_s|$ is maximum.
Assume that $u_1,\dots,u_k\in N(j)$ and $u_{k+1},\dots,u_s\notin N(j)$.
Since $j\notin V(H)$,
we have $k\leq s-2.$
Let $G'$ be obtained from $G^*$ by adding the edges $ju_{k+1},\dots,ju_{s-1}$.
Then $G'[V(H)\cup \{j\}]$ is $s$-clique connected, 
and by Lemma \ref{Lemma3},
we have $\rho_s(G')\ge \rho_s\bigl(G'[V(H)\cup\{j\}]\bigr)>\rho_s(H)=\rho_s(G^*).$

We claim that $G'$ is $tK_{r+1}$-free.
Otherwise,
there exists a copy $F_0$ of $K_{r+1}$ containing an edge \(ju_q\) for some $q\in \{k+1,\dots,s-1\}$.
Since $F_0\cong K_{r+1}$ and $s\le r$,
the edge $ju_q$ is contained in a copy $F$ of $K_{s+1}$ in $F_0$.
Let $c_s'=V(F)\setminus\{j\}.$
Then $c_s'$ is an $s$-clique of $G^*$.
Moreover,
$u_q\in c_s'\cap c_s$.
Hence, $c_s'\in C_s(H)$.
In $G'$,
every vertex of $c_s'$ is adjacent to $j$.
The new added edges incident with $j$ are $ju_{k+1},ju_{k+2},\dots,ju_{s-1}.$
Thus at most $s-1-k$ vertices of $c_s'$ are not adjacent to $j$ in $G^*$.
Therefore $|N(j)\cap c_s'|\ge s-(s-1-k)=k+1,$
which contradicts the choice of $c_s$.
Therefore, 
$G'$ is $tK_{r+1}$-free.
Since every copy of $K^+(n_1,\dots,n_r)$ contains a copy of $K_{r+1}$,
the graph $G'$ is also $tK^+(n_1,\dots,n_r)$-free and $\rho_s(G')>\rho_s(G^*)$,
which contradicts the maximality of $G^*$.
Hence,
$G^*$ is $s$-clique connected.
\end{proof}

By Lemma \ref{322},
the graph $G^*$ is $s$-clique connected.
Thus,
the $s$-clique tensor $\mathcal{A}_s(G^*)$ is weakly irreducible by Lemma \ref{dangqiejindang}.
By Lemma \ref{Lemma2},
there exists a positive eigenvector  corresponding to $\rho_s(G^*)$.

\begin{lemma}\label{310}
For vertex $v\in V (G^*)$ and the positive eigenvector $x = (x_1,x_2,\dots, x_n)^{\top}$ corresponding to $\rho_s(G^*)$ with $\max_{v\in V(G^*)}x_v=1$,
we have 
$$x_v^{s-1}\geq \frac{(s-1)!}{16}{r-1\choose s-1}\frac{1}{r^{s-1}}.$$
\end{lemma}

\begin{proof}
Recall that $x_{v_0}=\max\{x_v\ | \ v\in V(G^*)\setminus W\}$.
Suppose that $v_0\in V_k,k\in \{1,2,\dots,r\}$.
By Lemma \ref{377},
we have $x_{v_0}^{s-1}\geq
\frac{(s-1)!}{4}\binom{r-1}{s-1}\frac{1}{r^{s-1}}$.

Suppose to the contrary that there exists a vertex $u\in V(G^*)$ such that $x_u^{s-1}<\frac{(s-1)!}{16}{r-1\choose s-1}(\frac{1}{r})^{s-1}$,
then $x_u^{s-1}<\frac{x_{v_0}^{s-1}}{4}$.
Construct a graph $G'$ from $G^*$ by deleting all edges incident with $u$,
and then joining $u$ to every vertex in $\left(\bigcup_{\substack{i=1, i\neq k}}^r I_i\right)\setminus \{u\}.$
Similar to the proof of Lemma \ref{3991},
$G'$ is $tK^+(n_1,n_2,\dots,n_r)$-free.
Moreover,
for sufficiently large $n$,
$\left(\mathcal A_s(G')x^{s-1}\right)_u\geq\rho_s(G^*)x_{v_0}^{s-1}-(\xi^2+3\xi^3)n^{s-1}-O(n^{s-2}).$
Since $\xi$ is sufficiently small and $n$ is sufficiently large,
we have $$(\xi^2+3\xi^3)n^{s-1}+O(n^{s-2})\leq\frac{1}{2}\rho_s(G^*)x_{v_0}^{s-1}.$$
Hence $\left(\mathcal A_s(G')x^{s-1}\right)_u\geq\frac{1}{2}\rho_s(G^*)x_{v_0}^{s-1}>\rho_s(G^*)x_u^{s-1}=\left(\mathcal A_s(G^*)x^{s-1}\right)_u.$
Since $G'$ is obtained from $G^*$ by changing edges incident with $u$,
all changed $s$-cliques contain $u$.
Hence
$$x^{\top}\bigl(\mathcal A_s(G')-\mathcal A_s(G^*)\bigr)x^{s-1}=s x_u\Big(\left(\mathcal A_s(G')x^{s-1}\right)_u-\left(\mathcal A_s(G^*)x^{s-1}\right)_u\Big).$$
Since $x_u>0$,
it follows that $x^{\top}\bigl(\mathcal A_s(G')-\mathcal A_s(G^*)\bigr)x^{s-1}>0.$
Therefore,
\begin{align*}
\rho_s(G')-\rho_s(G^*)&\ge
\frac{x^{\top}\bigl(\mathcal A_s(G')-\mathcal A_s(G^*)\bigr)x^{s-1}}{||x||_s^s}>0.
\end{align*}
Thus $\rho_s(G')>\rho_s(G^*)$,
which contradicts the maximality of $G^*$.
Thus for each $v\in V(G^*)$ we have $x_v^{s-1}\geq \frac{(s-1)!}{16}{r-1\choose s-1}(\frac{1}{r})^{s-1}$ .
\end{proof}

\begin{lemma}\label{355}
For each $i\in \{1,2,\dots,r\}$,
we have $\Delta(G^*[V_i\setminus W])<th$.
\end{lemma}

\begin{proof}
Since $L\subseteq W$,
we have $V_i\setminus W=V_i\setminus(W\cup L)$ for each $i$.
Suppose to the contrary that there exists $i\in \{1,2,\dots,r\}$ such that $\Delta(G^*[V_i\setminus (W)])\geq th$.
Without loss of generality,
assume that there exists a vertex $v\in V_1\setminus W$ such that $d_{V_1\setminus W}(v)\geq th$.
Let $G'$ be obtained from $G^*$ by adding edges between $v$ and every vertex in $V_1\setminus W$ which is not adjacent to $v$.
We claim that $G'$ is $tK^+(n_1,n_2,\dots,n_r)$-free.
Otherwise,
$G'$ contains $t$ vertex-disjoint copies $F_1,F_2,\dots,F_t$ of $K^+(n_1,n_2,\dots,n_r)$.
Since $G^*$ is $tK^+(n_1,n_2,\dots,n_r)$-free,
there exists a $K^+(n_1,n_2,\dots,n_r)$ in $G'$ containing the new edge $vz$,
where $z\in V_1\setminus W$.
Without loss of generality,
assume that $vz\in E(F_1)$.
Let $S=\bigcup_{j=2}^t V(F_j).$
Then $|S|\leq (t-1)h.$
Since $d_{V_1\setminus W}(v)\geq th>|S|,$
there exists a vertex $u\in N(v)\cap (V_1\setminus W)\setminus S.$
Thus $uv$ is an edge of $G^*[V_1\setminus (W\cup S)].$
By Lemma \ref{344}, $G^*-(W\cup S)$ contains a copy $F$ of $K^+(n_1,\dots,n_r)$.
Therefore $F,F_2,\dots,F_t$ are $t$ disjoint copies of $K^+(n_1,\dots,n_r)$ in $G^*$,
a contradiction.
Hence $G'$ is $tK^+(n_1,\dots,n_r)$-free.

Now we compare $\rho_s(G^*)$ and $\rho_s(G')$.
The added edge $vz$ can be extended to a copy of $K^+(n_1,\dots,n_r)$ in $G'$.
Hence $G'$ contains an $s$-clique containing the new edge $vz$,
and this $s$-clique is not in $G^*$.
Then,
we have
$
\rho_s(G')-\rho_s(G^*)
>0,
$
a contradiction.
Therefore $\Delta(G^*[V_i\setminus W])<th$ for each $i\in\{1,2,\dots,r\}$.
\end{proof}

\begin{lemma}\label{313}
$|W|=t-1$,
and for each $i\in\{1,2,\dots,r\}$,
$V_i\setminus W$ is an independent set.
\end{lemma}

\begin{proof}
By the proof of Lemma \ref{312},
we have $|W|\leq t-1.$
Suppose that $|W|\leq t-2$,
and let $g=t-1-|W|$.
Clearly,
$g\geq 1$.
Choose a set  $S\subseteq V_1\setminus W$ with $|S|=g$,
and let $H=\bigcup_{i=1}^rG^*[V_i\setminus W]$.
Since $L\subseteq W$,
by Lemmas \ref{312} and \ref{355},
we have $\Delta(H)<th,\mu(H)\leq g$.
Hence,
by Lemma \ref{Hanson},
we have $e(H)\leq (th+1)g=O(1)$.
Let $G'$ be obtained from $G^*$ by deleting all edges of $H$ and then adding all edges between $S$ and $V_1\setminus (W\cup S)$.
Since $S\subseteq V_1\setminus W$,
we have $W\cap S=\emptyset$,
and
$|W\cup S|=|W|+|S|=|W|+g=t-1.$
We claim that $G'$ is $tK^+(n_1,n_2,\dots,n_r)$-free. $K^+(n_1,n_2,\dots,n_r)$ contains a copy of $K_{r+1}$. Thus every copy of
$K^+(n_1,n_2,\dots,n_r)$ in $G'$ contains two adjacent vertices containing in the
same part  of $V_1,V_2,\dots,V_r$.
By the construction of $G'$,
every edges of $G'$ with both vertices in the same $V_i$ has at least one vertex in $W\cup S$.
Hence,
every copy of $K^+(n_1,n_2,\dots,n_r)$ in $G'$ contains one vertex from $W\cup S$.
Since $|W\cup S|=t-1$,
the graph $G'$ contains at most $t-1$ disjoint copies of $K^+(n_1,n_2,\dots,n_r)$,
so $G'$ is $tK^+(n_1,n_2,\dots,n_r)$-free.

Recall that $x$ is a positive eigenvector
corresponding to $\rho_s(G^*)$ with $\max_{v\in V(G^*)}x_v=1$.
Since $e(H)=O(1)$ and every edge of $H$ is contained in at most $O(n^{s-2})$ copies of $K_s$,
deleting the edge in $H$ decreases $x^{\top} \mathcal{A}_s(G^*)x^{s-1}$ by at most  $O(n^{s-2})$.
On the other hand,
for every added edge $ab$ with $a\in S,b\in V_1\setminus (W\cup S)$,
by Lemma \ref{Cioaba} and Eq. (\ref{du}),
we know that $ab$ is contained in $\Omega(n^{s-2})$ copies of $K_s$ in $G'$.
Since the number of added edges is
$\Omega(n),
$
\(G'\), the added edges contribute \(\Omega(n^{s-1})\) in total.
By Lemma \ref{310},
we have $x^{\top} \mathcal{A}_s(G')x^{s-1}>x^{\top} \mathcal{A}_s(G^*)x^{s-1}$
for all sufficiently large $n$.
Hence $\rho_s(G')>\rho_s(G^*)$, a contradiction.
Hence,
$|W|=t-1$.

By Lemma \ref{312},
we have $\mu(\bigcup_{i=1}^rG^*[V_i\setminus W])\leq t-1-|W|=0.$
Hence,
$\bigcup_{i=1}^rG^*[V_i\setminus W]$ has no edges,
so $V_i\setminus W$ is an independent set for $i\in\{1,2,\dots,r\}$.
\end{proof}

\begin{lemma}\label{314}
For every vertex $w\in W,d_{G^*}(w)=n-1$.
\end{lemma}

\begin{proof}
Suppose to the contrary that there exist $w\in W$ and $v\in V(G^*)$ such that $wv\notin E(G^*)$.
Let $G'$ be obtained from $G^*$ by joining $w$ to every vertex in $V(G^*)\setminus\bigl(N(w)\cup\{w\}\bigr)$.
We first claim that $G'$ is $tK^+(n_1,n_2,\dots,n_r)$-free.
Otherwise,
if $G'$ contains $t$ pairwise vertex-disjoint copies of $K^+(n_1,n_2,\dots,n_r)$,
denoted by $F_1,\dots,F_t$.
Then there exists a $K^+(n_1,n_2,\dots,n_r)$,
say $F_1$,
contains the new edges.
Let $S=\bigcup _{i=2}^t V(F_i)$.
Then $|S|\leq (t-1)h$ and $w\notin S$.
By Lemma \ref{3112},
$G^*-\big((W\cup S)\setminus \{w\}\big)$ contains a copy of $K^+(n_1,n_2,\dots,n_r)$,
denoted by $F'$,
and $S\cap V(F')=\emptyset$.
Thus,
$F',F_2,\dots,F_t$ is a copy of $tK^+(n_1,n_2,\dots,n_r)$ in $G^*$,
a contradiction.
Hence,
$G'$ is $tK^+(n_1,n_2,\dots,n_r)$-free.

Let $y\in V(G^*)\setminus\bigl(N(w)\cup\{w\}\bigr)$.
Since $G^*$ is $s$-clique connected,
the vertex $y$ is contained in some $s$-cliques of $G^*$.
Let $c_{s-1}$ be an $(s-1)$-clique of $G^*$ containing $y$.
Then $c_{s-1}\cup\{w\}$ is a new $s$-clique of $G'$.
Therefore
%
%
by Lemma \ref{Lemma3},
we have $\rho_s(G')>\rho_s(G^*)$,
a contradiction.
\end{proof}

%

\begin{lemma}\label{315}
$G^*=K_{t-1}+K'$,
where $K'$ is a complete $r$-partite graph on $n-t+1$ vertices.
\end{lemma}

\begin{proof}
By Lemmas \ref{313} and  \ref{314},
we have $G^*=K_{t-1}+K'$,
where $V(K_{t-1})=W$ and $K'$ is an $r$-partite graph with parts $V_1\setminus W,\dots,V_r\setminus W$.
Suppose that $K'$ is not complete $r$-partite graph.
Then there exist vertices $u \in V_i \setminus W$ and $v\in V_j \setminus W$ with distinct $i , j$ such that $uv \notin E(G^*)$.
Let $G'$ be the graph obtained from $G^*$ by adding all edges $xy$ with $x\in V_i\setminus W$, $y\in V_j\setminus W$, $1\leq i<j\leq r$, 
and $xy\notin E(G^*)$.
Obviously,
$G'$ is $tK^+(n_1,n_2,\dots,n_r)$-free and forms some new $s$-cliques in $G'$.
Hence,
we have $\rho_s(G')>\rho_s(G^*)$,
a contradiction.
Hence,
$K'$ is a complete $r$-partite graph.
\end{proof}

We are now ready to give the proof of Theorem \ref{theorem1}.

\begin{proof}[The proof of Theorem \ref{theorem1}]
By Lemmas \ref{Xu} and \ref{315},
we obtain that if $G^*$ is a $tK^+(n_1,n_2,\\\dots,n_r)$-free graph with the maximum $s$-clique spectral radius, then $G^* = K_{t-1} + T_r(n - t + 1)$.
This completes the proof of Theorem \ref{theorem1}.
%
\end{proof}

At last we give the proof of Theorem \ref{theorem2}.

\begin{proof}[The proof of Theorem \ref{theorem2}]
Note that $F_i$ is a color-critical graph with $\chi(F_i)=r+1,1\leq i\leq t$.
Then,
for $1\leq i\leq t$,
there exists a graph $K^+(n^{(i)}_{1},n^{(i)}_{2},\dots,n^{(i)}_{r})$ containing a copy of $F_i$ as a subgraph.
Let $n_j=\max_{1\leq i\leq t}n_{j}^{(i)}$ for $1\leq j\leq r$.
Then we have $\bigcup _{i=1}^tF_i$ is a subgraph of $tK^+(n_1,n_2,\dots,n_r)$.
Hence,
${\rm spex}_s(n,\bigcup _{i=1}^tF_i)\leq {\rm spex}_s\big(n,tK^+(n_1,n_2,\dots,n_r)\big)$.

By Theorem \ref{theorem1},
we have ${\rm spex}_s\big(n,tK^+(n_1,n_2,\dots,n_r)\big)=\rho_s(K_{t-1}+T_r(n-t+1))$ and $K_{t-1}+T_r(n-t+1)$ is the unique graph with the maximum $s$-clique spectral radius among all $n$ vertices $tK^+(n_1,n_2,\dots,n_r)$-free graphs.
Since $K_{t-1}+T_r(n-t+1)$ is a $\bigcup_{i=1}^tF_i$-free graph with $n$ vertices,
we have ${\rm spex}_s(n,\bigcup _{i=1}^tF_i)\geq \rho_s\big(K_{t-1}+T_r(n-t+1)\big)$.
Hence,
$K_{t-1}+T_r(n-t+1)$ is the unique graph with the maximum $s$-clique spectral radius among all $n$ vertices $\bigcup_{i=1}^tF_i$-free graphs.
\end{proof}


\end{spacing}
\end{document}